\documentclass[12pt, reqno]{amsart}
\usepackage[utf8]{inputenc}
\usepackage{amsmath,amssymb,amsthm, mathrsfs}
\usepackage{indentfirst}
\usepackage{comment,relsize,braket}
\usepackage{enumerate}
\usepackage{graphicx}
\usepackage{color}
\usepackage[all]{xy}
\usepackage{fullpage}
\usepackage{tikz}
\usepackage{amsmath}
\usepackage{tikz}
\usepackage{tikz}
\usetikzlibrary{matrix,decorations.pathreplacing}
\usepackage[colorlinks, linkcolor=blue!70!black]{hyperref}
\usepackage[top=1in, bottom=1.2in, left=1in, right=1in]{geometry}
\usepackage{times}
\usepackage{tfrupee}
\usetikzlibrary{chains,fit}
\usetikzlibrary{shapes,snakes}
\usetikzlibrary{graphs}
\usepackage{graphicx,adjustbox}
\usepackage{hyperref}
\usepackage{amscd}

\usepackage{float}
\usepackage{caption}
\usepackage{subcaption}

\usepackage{cleveref}

\numberwithin{equation}{section}
\newtheorem{theorem}{Theorem}[section]
\newtheorem{proposition}[theorem]{Proposition}

\newtheorem{lemma}[theorem]{Lemma}
\newtheorem{corollary}[theorem]{Corollary}
\newtheorem{definition}[theorem]{Definition}

\newtheorem{remark}[theorem]{Remark}
\newtheorem{example}[theorem]{Example}
\newtheorem{question}[theorem]{Question}

\begin{document}

\title{Buchsbaumness of finite complement simplicial affine semigroups }

\author{OM PRAKASH BHARDWAJ}
		\address{Chennai Mathematical Institute, Siruseri, Chennai, Tamil Nadu 603103, India
			}
		\email{omprakash@cmi.ac.in; om.prakash@alumni.iitgn.ac.in}        
  
		\author{CARMELO CISTO}
		\address{Universit\`{a} di Messina, Dipartimento di Scienze Matematiche e Informatiche, Scienze Fisiche e Scienze della Terra\\
			Viale Ferdinando Stagno D'Alcontres 31\\
			98166 Messina, Italy}
		\email{carmelo.cisto@unime.it}
        

\begin{abstract}
    In this article, we classify all Buchsbaum simplicial affine semigroups whose complement in their (integer) rational polyhedral cone is finite. We show that such a semigroup is Buchsbaum if and only if its set of gaps is equal to its set of pseudo-Frobenius elements. Furthermore, we provide a complete structure of these affine semigroups. In the case of affine semigroups with maximal embedding dimension, we provide an explicit formula for the cardinality of the minimal presentation in terms of the number of extremal rays, the embedding dimension, and the genus. Finally, we observe that, unlike the complete intersection, Cohen–Macaulay, and Gorenstein properties, the Buchsbaum property is not preserved under gluing of affine semigroups. 
\end{abstract}

\maketitle

\section*{Introduction}

Let $A$ be a Noetherian local ring with maximal ideal $\mathfrak{m}$.  
If the difference $\ell_A(A/q) - e_A(q)$ is same for all parameter ideals $q$ of $A$, where $\ell_A(-)$ denotes the length and $e_A(q)$ is the multiplicity of $A$ relative to the parameter ideal $q$, then $A$ is called a \emph{Buchsbaum ring}. This is equivalent to the condition that every system of parameters of $A$ is a weak regular sequence. Thus, every Cohen-Macaulay ring is Buchsbaum. There has been very extensive study of Buchsbaum rings and singularities. For details, one can refer to the book by St\"{u}ckrad and Vogel~\cite{stuckrad-vogel}.
%In this article, we explore the Buchsbaum property for the affine semigroup rings.
There has been significant study of Buchsbaum semigroup rings, for example see \cite{bresinsky-schenzel-vogel, etocodimension2,buchsbaum-simplicial,goto2,goto, kamoibuchsbaum, lam-trung, trungbuchsbaum, trung-hoa} and many more. In~\cite{trungbuchsbaum}, Trung gave a criterion for the Buchsbaum property of a simplicial affine semigroup using the group generated by the affine semigroup and its extremal rays. A nice tool to study simplicial affine semigroups is the Apéry set with respect to its extremal rays.  
In~\cite{buchsbaum-simplicial}, the authors reformulated the criterion of Trung~\cite{trungbuchsbaum} in terms of the Ap\'{e}ry set with respect to its extremal rays.

A recent line of research, introduced for instance in \cite{failla2016algorithms} and \cite{garcia2018extension}, is focused on studying affine semigroups having finite complement in $\mathbb{N}^d$ or in the cone spanned by the semigroup, called $\mathcal{C}$-semigroups. The motivation for introducing this family of monoids lies in the aim of generalizing, in higher dimension, the concept of numerical semigroup in $\mathbb{N}$ and its features. Concerning the algebraic properties of the semigroup ring associated to a $\mathcal{C}$-semigroup in $\mathbb{N}^d$, it is known that if $d\geq 2$, then the semigroup ring is not Cohen-Macaulay. This fact closes the investigation, for $\mathcal{C}$-semigroups in $\mathbb{N}^d$ having $d\geq 2$, of some well known properties included in the rich theory of Cohen-Macaulay rings, like Gorenstein, Complete intersection and others (differently to the case of numerical semigroups rings, where a wide literature has been produced in this direction). So the road remains open for studying the more general property of being Buchsbaum for $\mathcal{C}$-semigroups, a research that has not yet been developed in detail for these affine semigroups. 

In this article, we characterize simplicial $\mathcal{C}$-semigrups having Buchsbaum semigroup ring. In particular, if $S$ is a $\mathcal{C}$-semigroups, where $\mathcal{C}$ is the integer cone spanned by $S$, then we prove in Theorem~\ref{thm:H=PF} that the semigroup ring of $S$ is Buchsbaum if and only if all elements in the finite set $\mathcal{C}\setminus S$ (namely, the gaps of $S$) are pseudo-Frobenius (that is, the sum with every nonzero element of $S$ belongs to $S$). This result allows also to characterize completely the structure of simplicial Buchsbaum $\mathcal{C}$-semigroups, in Corollary~\ref{cor:buchs-ideal}, by the notion of ideal. A classical problem in the study of semigroup rings is to determine the minimal number of generators of the defining ideal, denoted by $\mu(I_S)$. This problem has been extensively studied for numerical semigroups but remains less explored in the context of affine semigroups. When $S$ is a simplicial affine semigroup and $\mathbb{K}[S]$ is Cohen-Macaulay of codimension two, Kamoi \cite{kamoicm} proved that $\mu(I_S) \leq 3$. Eto \cite{etocodimension2} later generalized Kamoi’s result to the Buchsbaum case, showing that if $\mathbb{K}[S]$ is Buchsbaum of codimension two, then $\mu(I_S) \leq 4$. In \cite{goto}, Goto studied Buchsbaum rings of maximal embedding dimension. In our work, we focus on Buchsbaum simplicial affine semigroups of maximal embedding dimension with finite complement in their rational cone and provide an explicit formula for the minimal number of generators of the defining ideal of the associated semigroup ring.

The paper is structured as follows. In Section 1 we recall the basic notions about Buchsbaum property for affine semigroup rings and about $\mathcal{C}$-semigroups, used throughout this work. In Section 2, inspired by \cite{Li}, we provide some very useful tools concerning \emph{depth k regions} in an affine integer cone. These tools are a key ingredient to prove, in Section 3, the main result Theorem~\ref{thm:H=PF}, which characterizes Buchsbaum simplicial $\mathcal{C}$-semigroups by the equality of the set of gaps with the set of pseudo-Frobenius elements. In the same section, it is proved that a simplicial $\mathcal{C}$-semigroup $S$ is Buchsbaum if and only if $S\setminus \{\mathbf{0}\}$ is an ideal of $\mathcal{C}$. In Section 4, we explore the property of having maximal embedding dimension for Buchsbaum simplicial $\mathcal{C}$-semigroups. In particular, in Theorem~\ref{thm:buchs-MED}, we characterize these monoids in terms of their minimal elements with respect to the order induced by $\mathcal{C}$ and we provide, in Theorem~\ref{thm:mu-MED}, a formula for the cardinality of the minimal presentation in terms of the number of extremal rays, minimal generators and gaps. Finally, in Section 5, we conclude showing, by an example, that for affine semigroups the Buchsbaum property is not preserved by gluing and providing some open questions.

\section{Preliminaries}
Let $(A, \mathfrak{m})$ be a local ring. A sequence $x_1, \ldots, x_k$ in $A$ is called a \emph{weak-regular sequence} if
\[
\mathfrak{m}(x_1, \ldots, x_{i-1}) : x_i \subseteq (x_1, \ldots, x_{i-1})
\]
for each $i = 1, \ldots, k$. Let $d = \dim(A)$. A sequence $x_1, \ldots, x_d$ is called a \emph{system of parameters} of $A$ if the ideal $(x_1, \ldots, x_d)$ is $\mathfrak{m}$-primary. The ring $A$ is said to be \emph{Buchsbaum} if every system of parameters of $A$ is a weak-regular sequence. Suppose that $A$ is a finitely generated homogeneous algebra over a field, and let $\mathfrak{m}$ denote its maximal homogeneous ideal. In this case, $A$ is called \emph{Buchsbaum} if the local ring $A_{\mathfrak{m}}$ is Buchsbaum.

Let $S = \langle \mathbf{a}_1, \ldots, \mathbf{a}_n \rangle$ be an affine semigroup, minimally generated by $\mathbf{a}_1, \ldots, \mathbf{a}_n \in \mathbb{N}^d$. Throughout the paper we denote $\mathrm{msg}(S)=\{\mathbf{a}_1, \ldots, \mathbf{a}_n\}$, that is, the set of minimal generators of $S$. Consider the graded $\mathbb{K}$-algebra homomorphism
\[
\pi : \mathbb{K}[x_1, \ldots, x_n] \to \mathbb{K}[S] = \mathbb{K}[\mathbf{t}^{\mathbf{a}}:= t_1^{a_1}t_2^{a_2}\cdots t_d^{a_d}\mid \mathbf{a} \in S] \subseteq \mathbb{K}[t_1, \ldots, t_d]
\]
defined by $\pi(x_i) = \mathbf{t}^{\mathbf{a}_i}$ for each $i = 1, \ldots, n$, where $\mathbf{t}^{\mathbf{a}_i}$ denotes the monomial corresponding to $\mathbf{a}_i$. The map $\pi$ is surjective by construction, and its kernel, denoted by $I_S$, is a prime ideal of the polynomial ring $\mathbb{K}[x_1, \ldots, x_n]$. Therefore, the semigroup ring $\mathbb{K}[S]$ has the form
\[
\mathbb{K}[S] \cong \mathbb{K}[x_1, \ldots, x_n]/I_S,
\]
where $I_S = \ker(\pi)$ is referred to as the \emph{defining ideal} of the semigroup ring $\mathbb{K}[S]$. Define $\deg(x_i) = \mathbf{a}_i$ for $1 \leq i \leq n$. For a monomial $\mathbf{x}^\mathbf{u} := x_1^{u_1} x_2^{u_2} \cdots x_n^{u_n}$, define the $S$-degree of $\mathbf{x}^\mathbf{u}$ as
$\deg_S(\mathbf{x}^\mathbf{u}) := \sum_{i=1}^n u_i \mathbf{a}_i$. With this grading, the polynomial ring $\mathbb{K}[x_1, \ldots, x_n]$ becomes an $S$-graded ring, and the defining ideal $I_S$ becomes a graded ideal with respect to this grading. Thus the semigroup ring $\mathbb{K}[S]$ is a graded ring with respect to $S$-grading. We say $S$ is Buchsbaum if $\mathbb{K}[S]$ is a Buchsbaum ring. The defining ideal of the affine semigroup $S$ is related to the concept of presentation of a semigroup. Suppose $S$ is defined as above and consider the following map:
$$ \varphi: \mathbb{N}^n \rightarrow S \qquad \text{defined by} \qquad (b_1,\ldots,b_n)\mapsto \sum_{i=1}^n b_i \mathbf{a}_i.$$
The set $\ker(\varphi)=\{(\mathbf{a},\mathbf{b})\in \mathbb{N}^n \times \mathbb{N}^n\mid \varphi(\mathbf{a})=\varphi(\mathbf{b})\}$ is a congruence in $\mathbb{N}^n$ and it is called the \emph{kernel congruence} of $S$. A \emph{presentation} $\sigma$ of $S$ is a generating system of $\ker(\varphi)$, that is, $\ker(\varphi)$ is the minimal congruence containing $\sigma$. If $\sigma$ is a presentation of $S$, considering the defining ideal $I_S$, it is known that $I_S=(\{\mathbf{x}^{\mathbf{u}}-\mathbf{x}^{\mathbf{v}} \mid (\mathbf{u},\mathbf{v})\in \sigma\})$. If $\sigma$ is a minimal set with respect to inclusion in the set of presentations of $S$, then $\sigma$ is called a \emph{minimal presentation} of $S$. It is known that every minimal presentation of $S$ has the same cardinality and equivalently every minimal set of generators of $I_S$ has the same cardinality. For some references and all details about these facts, see for instance \cite{overview} or the book \cite{fg}. By the above discussion, for every affine semigroup $S$ we define the invariant $\mu(I_S)$ as the cardinality of a minimal presentation of $S$ or, equivalently, the cardinality of a minimal generating set of the defining ideal $I_S$.

\medskip
Let $S\subseteq \mathbb{N}^d$ be an affine semigroup. We denote $\mathrm{cone}(S)=\mathrm{Span}_{\mathbb{Q}_{\geq 0}}(S)=\{\sum_{i=1}^r q_i \mathbf{s}_i\mid r\in \mathbb{N}, \mathbf{s}_i\in S, q_i\in \mathbb{Q}_{\geq 0}\}$. Let $A$ be the set of minimal generators of $S$, then there exists a subset $B\subseteq A$ such that $\mathrm{cone}(S)=\mathrm{Span}_{\mathbb{Q}_{\geq 0}}(B)$ and for all $B'\subsetneq B$ then $\mathrm{cone}(S)\neq \mathrm{Span}_{\mathbb{Q}_{\geq 0}}(B')$. A set having the property of $B$ is called a set of \emph{extremal rays} of $S$. Assume $B=\{\mathbf{b}_1,\ldots,\mathbf{b}_t\}$ for some $t\in \mathbb{N}$. For all $i\in \{1,\ldots,t\}$, let us denote $\mathbf{n}_i=\min \{q\mathbf{b}_i \in S\mid q\in \mathbb{Q}_{\geq 0}\}$. Define $E=\{\mathbf{n}_1,\ldots,\mathbf{n}_t\}$. Obviously, $E$ is a set of extremal rays of $S$ that we call \textit{minimal extremal rays}. Denote by $\mathrm{G}(S)=\{\mathbf{a}-\mathbf{b}\in \mathbb{Z}^d\mid \mathbf{a},\mathbf{b}\in S\}$ the subgroup in $\mathbb{Z}^d$ generated by $S$.  If $S$ is an affine semigroup having $t$ extremal rays and $t=\mathrm{rank}(\operatorname{G}(S))$, then $S$ is called a \textit{simplicial} affine semigroup. In this case, the vectors in a set of extremal rays of $S$ are linearly independent over $\mathbb{Z}$. If $S\subseteq \mathbb{N}^d$ is a simplicial affine semigroup, we can always assume $\mathrm{rank}(\operatorname{G}(S))=d$, otherwise it is possible to embed $S$ in $\mathbb{N}^r$, with $r<d$. So, throughout the paper, we will always assume that $S$ has exactly $d$ extremal rays.

\medskip
The set of gaps of $S$ is defined as $\mathcal{H}(S)=(\mathrm{cone}(S)\setminus S)\cap \mathbb{N}^d$ and the value $\operatorname{g}(S)=|\mathcal{H}(S)|$ is referred as the \emph{genus} of $S$.  An element $\mathbf{f}\in \mathcal{H}(S)$ is called a\textit{ pseudo-Frobenius} element if $\mathbf{f} + S\setminus\{0\} \subseteq S$. The set of pseudo-Frobenius elements is denoted by $\mathrm{PF}(S)$. In particular,
$$\operatorname{PF}(S)=\{\mathbf{h}\in \mathcal{H}(S)\mid \mathbf{h}+\mathbf{s}\in S, \text{ for all }\mathbf{s}\in S\setminus \{\mathbf{0}\}\}.$$
Let $S$ be an affine semigroup minimally generated by $\{\mathbf{a}_1, \ldots, \mathbf{a}_n\}$ with $\mathrm{PF}(S) \neq \emptyset$. Then, by \cite[Theorem~6]{garcia-pseudofrobenius}, the projective dimension of $\mathbb{K}[S]$ is $n-1$, which is the maximal possible projective dimension for a semigroup ring associated to an affine semigroup minimally generated by $n$ elements. For this reason, an affine semigroup with a non-empty set of pseudo-Frobenius elements is called a \emph{maximal projective dimension semigroup} (see \cite{garcia-pseudofrobenius}). Consequently, by Auslander-Buchsbaum formula, a maximal projective dimension semigroup is never Cohen-Macaulay, unless it is one-dimensional.

\medskip 
In \cite{trungbuchsbaum}, Trung gave the following necessary and sufficient condition for the Buchsbaum property of a semigroup ring associated to a simplicial affine semigroup.  

\begin{theorem}[{\cite[Lemma 4]{trungbuchsbaum}}]\label{simplicialCMBcriterion}
Let $S$ be a simiplicial affine semigroup in $\mathbb{N}^d$ and $\{\mathbf{a}_1,\ldots,\mathbf{a}_d\}$ be a set of extremal rays of $S.$ Then $S$ is Buchsbaum if and only if 
\[\{ \mathbf{x} \in G(S) \mid \mathbf{x}+2\mathbf{a}_i, \mathbf{x}+2\mathbf{a}_j \in S ~ \text{for some }~ i \neq j\} + S \setminus \{\mathbf{0}\} \subset S.\]
\end{theorem}

\noindent
Using the above criteria and denoting:
 $$\mathcal{D}(S) = \{\mathbf{a} \in \operatorname{G}(S) \setminus S \mid \mathbf{a} + 2\mathbf{a}_i , \mathbf{a} + 2\mathbf{a}_j \in S \text{ for some } i\neq j, 1 \leq i, j \leq d\},$$
 the authors in \cite{affineII} proved the following characterization for Buchsbaum simplicial maximal projective dimension semirgoups. 

 \begin{proposition}[{\cite[Proposition 3.5]{affineII}}] \label{prop:buchsbaum}
Let $S$ be a simplicial maximal projective dimension affine semigroup in $\mathbb{N}^d$, $d > 1$. Then $S$ is Buchsbaum if
and only if $\mathcal{D}(S) = \operatorname{PF}(S)$ for every set of extremal rays of $S$.
\end{proposition}

\medskip
For a subset $L \subseteq \mathbb{N}^d$, define the partial order $\leq_L$ on $\mathbb{N}^d$ as follows: $\mathbf{x} \leq_L \mathbf{y}$ if and only if $\mathbf{y}-\mathbf{x} \in L.$ If $A,B$ are subset of $\mathbb{N}^d$, we denote as usual $A+B=\{\mathbf{a}+\mathbf{b}\mid \mathbf{a}\in A, \mathbf{b}\in B\}$. Moreover, for $k\in \mathbb{N}$ we denote $kA=\{\mathbf{0}\}$ if $k=0$, and $kA=(k-1)A+A$. In particular, if $k\geq 1$, then $kA=\{\mathbf{a}_1+\ldots+\mathbf{a}_k\mid \mathbf{a}_j\in A\text{ for every }j\in \{1,\ldots,k\}\}$.

\medskip
When the complement of an affine semigroup $S$ in its integer cone $\mathcal{C}=\mathrm{cone}(S)\cap \mathbb{N}^d$ is finite, it is called $\mathcal{C}$-semigroup. Note that $\mathcal{C}$-semigroups are a particular case of maximal projective dimension semigroups. In fact, if $S$ is a $\mathcal{C}$-semigroup, then $\mathcal{H}(S)$ is finite, and the maximal gaps with respect to the order $\leq_S$ are pseudo-Frobenius elements.  When $\mathcal{C}=\mathbb{N}^d$, we call $S$ a \textit{generalized numerical semigroup}. While, in the case $d=1$, $S$ is called a \textit{numerical semigroup}.

\begin{remark} \rm
Let $B\subset \mathbb{N}^d$ a set a consider $\mathcal{C}_B=\mathrm{cone}(B)\cap \mathbb{N}^d$. That is, $\mathcal{C}_B$ is any affine integer cone in $\mathbb{N}^d$. Let $S\subseteq \mathbb{N}^d$ be an affine semigroup such that $S\subseteq \mathcal{C}_B$. If $\mathcal{C}_B\setminus S$ is a finite set, then $\mathcal{C}_B=\mathcal{C}=\mathrm{cone}(S)\cap \mathbb{N}^d$. In fact, it is easy to see that $\mathcal{C}\subseteq \mathcal{C}_B$. While, if $\mathbf{x}\in \mathcal{C}_B\setminus S$, then by the finiteness of $\mathcal{C}_B\setminus S$ there exists $\lambda\in \mathbb{N}$ such that $\lambda \mathbf{x}\in S$. So, $\lambda \mathbf{x}\in \mathcal{C}$ which implies $\mathbf{x}\in \mathcal{C}$.  Hence, if $S$ is an affine semigroup having finite complement in some affine integer cone, then $S$ has finite complement in the integer cone spanned by itself, in particular $S$ is a $\mathcal{C}$-semigroup. 
\end{remark}

\section{Multsets and depth $k$-regions}
The following definitions and properties regarding \emph{depth k regions} are very useful for our goal. They have been introduced in \cite{Li} in the context of generalized numerical semigroups and focusing on the order $\leq_{\mathbb{N}^d}$. By focusing on the order $\leq_\mathcal{C}$, we generalize these concepts to $\mathcal{C}$-semigroups. Recall that a set $A\subseteq \mathbb{N}^d$ is an \emph{antichain} with respect to some partial order $\leq_p$ if for all $\mathbf{a},\mathbf{b} \in A$,  $\mathbf{a} \nleq_p \mathbf{b}$ and $\mathbf{b} \nleq_p \mathbf{a}$.
 
\begin{definition}\rm
Let $\mathcal{C}$ be an integer cone and $\mathcal{M}\subseteq \mathcal{C}$. We say that $\mathcal{M}$ is a \emph{multset} of $\mathcal{C}$ if $\mathcal{M}$ is an antichain with respect to $\leq_{\mathcal{C}}$ and contains a set of extremal rays of $\mathcal{C}$. 
\end{definition}
 
Let $S\subseteq \mathcal{C}$, where $\mathcal{C}$ is an affine integer cone in $\mathbb{N}^d$. Define $\mathcal{M}(S)=\mathrm{Minimals}_{\leq_\mathcal{C}}(S\setminus \{\mathbf{0}\})$. It is not difficult to see that if $S$ is a $\mathcal{C}$-semigroup, then $\mathcal{M}(S)$ is a multset of $\mathcal{C}$. Moreover, in this case $\mathcal{M}(S)\subseteq \mathrm{msg}(S)$, which means $\mathcal{M}(S)$ is a finite set. 
 
 \begin{definition}\rm
Let $\mathcal{M}\subseteq \mathbb{N}^d$ be a multset of an integer cone $\mathcal{C}$. For every nonnegative integer $k$, define:
$$\mathcal{R}_{\leq k}(\mathcal{M})= \{\mathbf{x}\in \mathcal{C}\mid \mathbf{x}\ngeq_\mathcal{C} \mathbf{a},\text{ for every }\mathbf{a}\in k\mathcal{M}\}
$$ 
The set $\mathcal{R}_{k}(\mathcal{M})=\mathcal{R}_{\leq k}(\mathcal{M})\setminus \mathcal{R}_{\leq k-1}(\mathcal{M})$, with the convention $\mathcal{R}_{0}(\mathcal{M})=\{\mathbf{0}\}$, is called \emph{depth k-region}.
 \end{definition}

 \begin{example}\rm 
Consider the affine integer cone $\mathcal{C}=\mathrm{Span}_{\mathbb{Q}_{\geq 0}}((1,1),(3,1))\cap \mathbb{N}^2$. The set $\mathcal{M}=\{(3,3),(6,2),(3,2)\}$ is a multset of $\mathcal{C}$, in particular $\{(3,3),(6,2)\}$ is a set of (not minimal) extremal rays of $\mathcal{C}$. In this case, for instance, it is possible to compute 
$$\mathcal{R}_1(\mathcal{M})=\{(1,1),(2,1),(3,1),(2,2),(4,2),(5,2)\}.$$ Some other examples of depth $k$-region, for the cone $\mathcal{C}$ and multset $\mathcal{M}$, are shown in Figure~\ref{fig:regions}.
\begin{figure}[h!]
\includegraphics[scale=0.7]{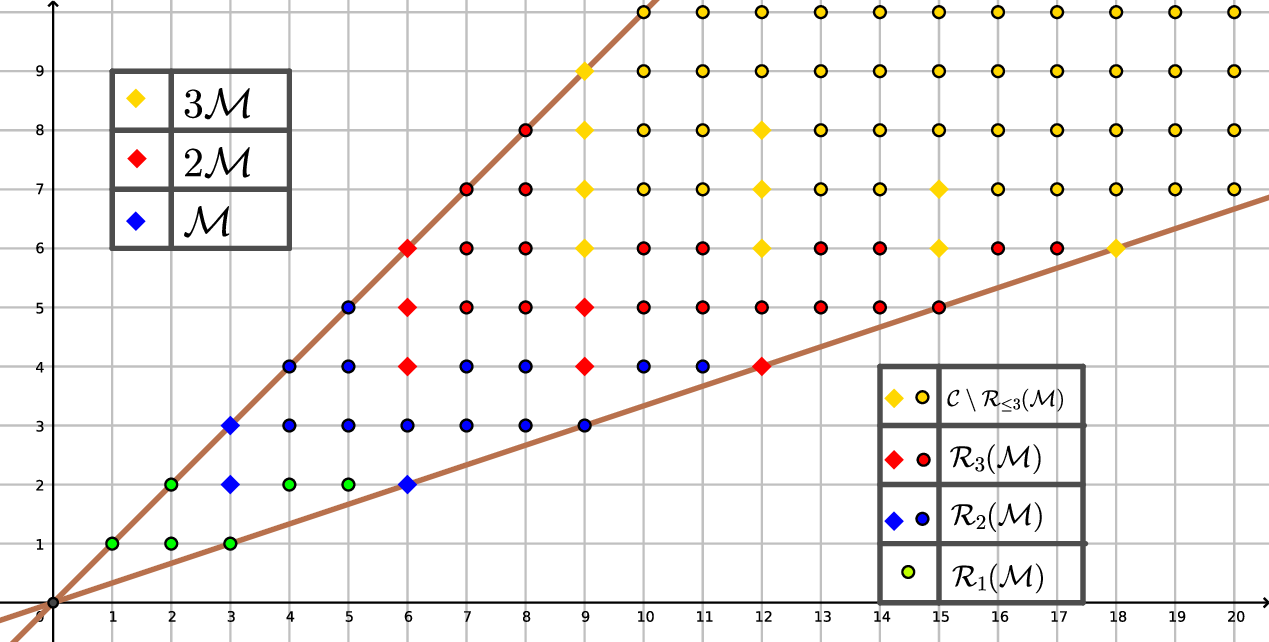}
\caption{Examples of $\mathcal{R}_k(\mathcal{M})$ in $\mathcal{C}=\mathrm{Span}_{\mathbb{Q}_{\geq 0}}((1,1),(3,1))\cap \mathbb{N}^2$}
\label{fig:regions}
\end{figure}
\end{example} 
 
 Let $\mathcal{C}\subseteq \mathbb{N}^d$ be an affine integer cone. In order to prove the next result, observe that if $\mathbf{v},\mathbf{w} \in \mathcal{C}$ such that $\mathbf{v}\leq_\mathcal{C} \mathbf{w}$, then $\mathbf{v}\leq_{\mathbb{N}^d} \mathbf{w}$. Moreover, recall that a term order $\preceq$ in $\mathbb{N}^d$ is a total order such that for all $\mathbf{v}, \mathbf{w},\mathbf{z}\in \mathbb{N}^d$ then $\mathbf{0}\preceq \mathbf{v}$ and if $\mathbf{v}\preceq \mathbf{w}$, then $\mathbf{v}+\mathbf{z}\preceq \mathbf{w}+\mathbf{z}$. It is known that if $\mathbf{v}\leq_{\mathbb{N}^d} \mathbf{w}$, then $\mathbf{v}\preceq \mathbf{w}$. In particular, if $\mathbf{v}\leq_\mathcal{C} \mathbf{w}$ then $\mathbf{v}\preceq \mathbf{w}$.
 
 \begin{proposition}\label{prop:unionR}
 Let $\mathcal{M}\subseteq \mathbb{N}^d$ be a multset of an integer cone $\mathcal{C}$. Then the following holds:
 \begin{enumerate}
 \item $\mathcal{R}_{\leq k}(\mathcal{M})\subsetneq \mathcal{R}_{\leq k+1}(\mathcal{M})$ for every $k\in \mathbb{N}\setminus \{0\}$.
 \item $\mathcal{C}=\bigsqcup_{k\geq 0}\mathcal{R}_{ k}(\mathcal{M})$ (disjoint union).
 \end{enumerate}
 \end{proposition}
 \begin{proof}
 (1) Let $\mathbf{x}\in \mathcal{R}_{\leq k}(\mathcal{M})$ and assume $\mathbf{x}\notin \mathcal{R}_{\leq k+1}(\mathcal{M})$. So, there exists  $\mathbf{m}_1+\cdots \mathbf{m}_{k+1}\in \mathcal{M}$ such that $\mathbf{x}\geq_\mathcal{C} \mathbf{m}_1+\cdots \mathbf{m}_{k+1}$. In particular $\mathbf{x}\geq_\mathcal{C} \mathbf{m}_1+\cdots \mathbf{m}_{k}$, which means $\mathbf{x}\notin \mathcal{R}_{\leq k}(\mathcal{M})$, a contradiction. Therefore, $\mathbf{x}\in \mathcal{R}_{\leq k+1}(\mathcal{M})$. In order to prove the proper inclusion, let $\preceq$ be a term order in $\mathbb{N}^d$ and define $\mathbf{m}=\min_\preceq(\mathcal{M})$. Let $k\in \mathbb{N}$.  It is obvious that $k\mathbf{m}\notin \mathcal{R}_{\leq k}(\mathcal{M})$, we show that $k\mathbf{m}\in \mathcal{R}_{\leq k+1}(\mathcal{M})$. If the opposite holds, then $k\mathbf{m}\geq_\mathcal{C} \mathbf{m}_1+\cdots \mathbf{m}_{k+1}$ for some $\mathbf{m}_1,\ldots,\mathbf{m}_{k+1}\in \mathcal{M}$. In particular, $k\mathbf{m}\geq_{\mathbb{N}^d} \mathbf{m}_1+\cdots \mathbf{m}_{k+1}$, but by the definition of $\mathbf{m}$ we have $k\mathbf{m}\prec \mathbf{m}_1+\cdots \mathbf{m}_{k+1}$, a contradiction. So, $k\mathbf{m}\in \mathcal{R}_{\leq k+1}(\mathcal{M})$. 
 
\noindent  (2) The inclusion ``$\supseteq$'' is trivial. For the other inclusion, let $\mathbf{x}\in \mathcal{C}$. Let us denote $\operatorname{B}(\mathbf{x})=\{\mathbf{y}\in \mathcal{C}\mid \mathbf{x}\geq_\mathcal{C} \mathbf{y}\}$. We easily deduce that $\mathbf{x}\in \mathcal{R}_{\leq k}(\mathcal{M})$ if and only if $k\mathcal{M}\cap \operatorname{B}(\mathbf{x})=\emptyset$. Observe that $\operatorname{B}(\mathbf{x})\subseteq \{\mathbf{y}\in \mathcal{C}\mid \mathbf{x}\geq_{\mathbb{N}^d} \mathbf{y}\}$. So, the set $\operatorname{B}(\mathbf{x})$ is finite. Since $\bigcup_{k\geq 0}k\mathcal{M}$ is an infinite set, we argue that there exist $\overline{k}\in \mathbb{N}$ such that $\overline{k}\mathcal{M}\cap \operatorname{B}(\textbf{x})=\emptyset$. Hence, if $\ell=\min\{k\in \mathbb{N}\mid k\mathcal{M}\cap \operatorname{B}(\textbf{x})=\emptyset\}$, we obtain $\mathbf{x}\in \mathcal{R}_{\leq \ell}(\mathcal{M})\setminus \mathcal{R}_{\leq \ell-1}(\mathcal{M})= \mathcal{R}_{\ell}(\mathcal{M})$, and the inclusion is proved.
 \end{proof}
 
 \begin{lemma}\label{lem:claim1}
 Let $\mathcal{M}\subseteq \mathbb{N}^d$ be a multset of an integer cone $\mathcal{C}$. If $\mathbf{x}\in \mathcal{R}_{k}(\mathcal{M})$ and $\mathbf{y}\in  \mathcal{R}_{\ell}(\mathcal{M})$, then $\mathbf{x}+\mathbf{y}\notin \mathcal{R}_{k+\ell-2}(\mathcal{M})$. In particular, $\mathbf{x}+\mathbf{y}\in \mathcal{R}_{n}(\mathcal{M})$ for some $n\geq k+\ell-1$.
 \end{lemma}
 \begin{proof}
 By $\mathbf{x}\in \mathcal{R}_{k}(\mathcal{M})$ and $\mathbf{y}\in  \mathcal{R}_{\ell}(\mathcal{M})$, we have $\mathbf{x}\notin \mathcal{R}_{\leq k-1}(\mathcal{M})$ and $\mathbf{y}\notin  \mathcal{R}_{\leq \ell-1}(\mathcal{M})$. So, there exist $\mathbf{a}\in (k-1)\mathcal{M}$ and $\mathbf{b}\in (\ell-1)\mathcal{M}$ such that $\mathbf{x}\geq_\mathcal{C}\mathbf{a}$ and $\mathbf{y}\geq_\mathcal{C}\mathbf{b}$. Therefore, $\mathbf{x}+\mathbf{y}\geq_\mathcal{C} \mathbf{a}+\mathbf{b}\in (k+\ell-2)\mathcal{M}$, which means $\mathbf{x}+\mathbf{y}\notin \mathcal{R}_{k+\ell-2}(\mathcal{M})$. The last claim easily follows from Proposition~\ref{prop:unionR}.
 \end{proof}

\begin{lemma}\label{lem:claim2}
Let $\mathcal{M}\subseteq \mathbb{N}^d$ be a multset of an integer cone $\mathcal{C}$. If $\mathbf{x}\in \mathcal{R}_{k}(\mathcal{M})$ with $k\geq 2$, then there exists $\mathbf{m}\in \mathcal{M}$ such that $\mathbf{m}\leq_\mathcal{C}\mathbf{x}$ and $\mathbf{x}-\mathbf{m}\in \mathcal{R}_{k-1}(\mathcal{M})$.
\end{lemma}
\begin{proof}
By $\mathbf{x} \in \mathcal{R}_{k}(\mathcal{M})$ we have $\mathbf{x}\notin \mathcal{R}_{\leq k-1}(\mathcal{M})$. So, $\mathbf{x}\geq_{\mathcal{C}} \mathbf{m}_1+\cdots+\mathbf{m}_{k-1}$, with $\mathbf{m}_j\in \mathcal{M}$ for every $j\in \{1,\ldots,k-1\}$. Hence, we deduce $\mathbf{x}-\mathbf{m}_1\notin \mathcal{R}_{\leq k-2}(\mathcal{M})$. In order to conclude, it is sufficient to prove $\mathbf{x}-\mathbf{m}_1\in \mathcal{R}_{k-1}(\mathcal{M})$. If we assume the opposite, then $\mathbf{x}-\mathbf{m}_1\geq_{\mathcal{C}} \mathbf{n}_1+\cdots \mathbf{n}_{k-1}$ with $\mathbf{n}_j\in\mathcal{M}$ for every $j\in\{1,\ldots,k-1\}$. Hence, $\mathbf{x}\geq_\mathcal{C}\mathbf{m}_1+\mathbf{n}_1+\cdots+\mathbf{n}_{k-1}$, which means $\mathbf{x}\notin \mathcal{R}_{\leq k}(\mathcal{M})$, a contradiction with $\mathbf{x}\in \mathcal{R}_{k}(\mathcal{M})$. Therefore, $\mathbf{x}-\mathbf{m}_1\in \mathcal{R}_{\leq k-1}(\mathcal{M})\setminus \mathcal{R}_{\leq k-2}(\mathcal{M})=\mathcal{R}_{k-1}(\mathcal{M})$ and we can conclude.
\end{proof}

\begin{lemma}\label{lem:claim3}
Let $S\subseteq \mathbb{N}^d$ be $\mathcal{C}$-semigroup and let $\mathcal{M}=\mathcal{M}(S)$. If $\mathcal{R}_{k}(\mathcal{M})\subseteq S$ for some $k\in \mathbb{N}$, then $\mathcal{R}_{k+1}(\mathcal{M})\subseteq S$.
\end{lemma}
\begin{proof}
Suppose $\mathcal{R}_{k}(\mathcal{M})\subseteq S$ for some $k\in \mathbb{N}$ and let $\mathbf{x}\in \mathcal{R}_{k+1}(\mathcal{M})$. By Lemma~\ref{lem:claim2}, we have $\mathbf{x}-\mathbf{m}\in \mathcal{R}_{k}(\mathcal{M})\subseteq S$ for some $\mathbf{m}\in \mathcal{M}$. Hence $\mathbf{x}=(\mathbf{x}-\mathbf{m})+\mathbf{m}\in S$, and the result is proved.
\end{proof}

Let $S\subseteq \mathbb{N}^d$ be a $\mathcal{C}$-semigroup and $\mathcal{M}=\mathcal{M}(S)$. Since $\mathcal{C}\setminus S$ is finite, by Proposition~\ref{prop:unionR} there exists $k\in \mathbb{N}$ such that $\mathcal{R}_{k}(\mathcal{M})\subseteq S$ and in such a case $\mathcal{R}_{\ell}(\mathcal{M})\subseteq S$ for every $\ell \geq k$. Considering the integer $q=\min\{k\in\mathbb{N}\mid \mathcal{R}_{k+1}(\mathcal{M})\subseteq S\}$, we define $\mathrm{Depth}(S)=q$.

\begin{proposition}\label{prop:not-PF}
Let $S\subseteq \mathbb{N}^d$ be a $\mathcal{C}$-semigroup. Suppose $\mathcal{H}(S)\neq \operatorname{PF}(S)$. Then there exists $\mathbf{h}\in \mathcal{H}(S)\setminus \operatorname{PF}(S)$ such that $\mathbf{h}+k\mathbf{s}\in S$ for all $\mathbf{s}\in S$ and for every $k > 1$.
\end{proposition}
\begin{proof}
It is sufficient to prove the result for $k=2$. Let $\mathbf{g}\in \mathcal{H}(S)\setminus \operatorname{PF}(S)$. So, there exists $\mathbf{s}\in S\setminus \{\mathbf{0}\}$ such that $\mathbf{g}+\mathbf{s}=\mathbf{f}\in \operatorname{PF}(S)$. In particular, there exists $\mathbf{n}\in \mathcal{M}(S)$ such that $\mathbf{n}\leq_\mathcal{C}\mathbf{s}\leq_\mathcal{C} \mathbf{f}$. Let $q=\mathrm{Depth}(S)$. Since $\mathbf{f}\geq_\mathcal{C} \mathbf{n}$, we have $q\geq 2$. Let us denote $\mathcal{M}=\mathcal{M}(S)$. By definition of $q$, we have $\mathcal{R}_{q+1}(\mathcal{M})\subseteq S$ and $\mathcal{R}_{q}(\mathcal{M})\cap \mathcal{H}(S)\neq \emptyset$. Let $\mathbf{h}'\in \mathcal{R}_{q}(\mathcal{M})\cap \mathcal{H}(S)$. Since $q\geq 2$, by Lemma~\ref{lem:claim2} there exists $\mathbf{m}'\in \mathcal{M}(S)$ such that $\mathbf{m}'\leq_\mathcal{C} \mathbf{h}'$ and $\mathbf{h}'-\mathbf{m}'\in \mathcal{R}_{q-1}(\mathcal{M})$. Consider that $\mathbf{h}=\mathbf{h}'-\mathbf{m}'\in \mathcal{H}(S)\setminus \operatorname{PF}(S)$. Moreover, for every $\mathbf{m}\in \mathcal{M}$, observing that $\mathbf{m}\in \mathcal{R}_{2}(\mathcal{M})$,  by $\mathbf{h}'-\mathbf{m}'\in \mathcal{R}_{q-1}(\mathcal{M})$ and Lemma~\ref{lem:claim1} we obtain $\mathbf{h}+2\mathbf{m}=(\mathbf{h}'-\mathbf{m}')+\mathbf{m}+\mathbf{m}\in \mathcal{R}_{n}(\mathcal{M})$ with $n>q$. So, for all $\mathbf{m}\in \mathcal{M}(S)$, we proved that $\mathbf{h}+2\mathbf{m}\in \mathcal{R}_{n}(\mathcal{M})$ with $n>q$. Now let $\mathbf{s}\in S$. We can assume $\mathbf{s}=\mathbf{n}+\mathbf{c}$ with $\mathbf{n}\in \mathcal{M}$ and $\mathbf{c}\in \mathcal{C}$. Hence, we have $\mathbf{h}+2\mathbf{s}=\mathbf{h}+2\mathbf{n}+2\mathbf{c}\in \mathcal{R}_{n}(\mathcal{M})$ with $n>q$. Since $\mathcal{R}_{q+1}(\mathcal{M})\subseteq S$, by Lemma~\ref{lem:claim3} we obtain $\mathbf{h}+2\mathbf{s}\in S$ and this concludes the proof. 
\end{proof}

\section{Buchsbaum simplicial $\mathcal{C}$-semigroups}

\medskip
Let S be a simplicial affine semigroup minimally generated by $\{\mathbf{a}_1,\ldots, \mathbf{a}_d , \mathbf{a}_{d+1},\ldots, \mathbf{a}_n \}\subseteq \mathbb{N}^d$,
and let $\operatorname{G}(S)$ be the group generated by $S$ in $\mathbb{Z}^d$ . Assume that $\{\mathbf{a}_1,\ldots, \mathbf{a}_d \}$ is a set of extremal rays of
$S$. Recall the set
$$\mathcal{D}(S) = \{\mathbf{a} \in \operatorname{G}(S) \setminus S \mid \mathbf{a} + 2\mathbf{a}_i , \mathbf{a} + 2\mathbf{a}_j \in S \text{ for some } i\neq j, 1 \leq i, j \leq d\}.$$

Actually, if $S$ is a simplicial affine semigroup, then $\mathcal{D}(S)\subseteq \mathcal{H}(S)$, as deduced by the following result.

\begin{proposition}\label{prop:outside-cone}
    Let $\mathcal{C}$ be a simplicial rational polyhedral cone inside $\mathbb{Q}^d_{\geq 0}$ with extremal rays $\{\mathbf{a}_1, \ldots, \mathbf{a}_d\}$. Let $\mathbf{x} \in \mathbb{Q}^d$ such that $\mathbf{x} + k\mathbf{a}_i, \mathbf{x}+ k\mathbf{a}_j \in \mathcal{C}$ for some $k \geq 1$ and $i \neq j$, then $\mathbf{x} \in \mathcal{C}$. 
\end{proposition}
\begin{proof}
    Since $\mathcal{C}$ is a simplicial polyhedral cone, we can write it as the intersection of $d$-closed halfspaces (see, for instance, \cite[Chapter 1]{bruns-gub}). Let $\mathcal{C} = H_1^+ \cap \ldots \cap H_d^+.$ Let $h_i(x) = 0$ defines the supporting hyperplane corresponding to the halfspace $H_i^+$. Suppose $\mathbf{x} \notin \mathcal{C}$, then there exists $q \in [1,d]$ such that $h_q(\mathbf{x}) < 0$. Since $\mathcal{C}$ is a simplicial, we know that each facet of $\mathcal{C}$ contains exactly $d-1$ extremal rays. Therefore, there exist $\{\mathbf{a}_{i_1}, \ldots, \mathbf{a}_{i_{d-1}}\} \subseteq \{\mathbf{a}_1, \ldots, \mathbf{a}_d\}$ such that $h_q(\mathbf{a}_{i_j}) = 0$ for all $j \in [1,d-1]$. Therefore, we have $h_q(\mathbf{x} + k\mathbf{a}_{i_j})= h_q(\mathbf{x}) + kh_q(\mathbf{a}_{i_j})=h_q(\mathbf{x}) < 0$ for all $j \in[1,d-1]$. Thus, $\mathbf{x} + k\mathbf{a}_{i_j} \notin \mathcal{C}$ for all $j \in [1,d-1]$. This is a contradiction to $\mathbf{x} + k\mathbf{a}_i, \mathbf{x}+ k\mathbf{a}_j \in \mathcal{C}$ and $i \neq j$. Hence, $\mathbf{x} \in \mathcal{C}$. 
\end{proof}

Now, having Proposition~\ref{prop:buchsbaum} in mind and observing that every $\mathcal{C}$-semigroup is a maximal projective dimension semigroup, we give a more precise characterization for simplicial $\mathcal{C}$-semigroups in terms of pseudo-Frobenius elements and gapset.

\begin{theorem}\label{thm:H=PF}
    Let $S\subseteq \mathbb{N}^d$ be a simplicial $\mathcal{C}$-semigroup with $d>1$. Then $S$ is Buchsbaum if and only if $\mathcal{H}(S)=\operatorname{PF}(S)$.
\end{theorem}
\begin{proof}
  Assume $\{\mathbf{a}_1,\ldots,\mathbf{a}_d\}$ is a set of extremal rays of $S$. Suppose $S$ is Buchsbaum. Then, using Proposition~\ref{prop:buchsbaum}, we have $\operatorname{PF}(S)=\mathcal{D}(S)$. Assuming $\mathcal{H}(S)\neq \operatorname{PF}(S)$, by Proposition~\ref{prop:not-PF} there exists $\mathbf{h}\in \mathcal{H}(S)\setminus \operatorname{PF}(S)$ such that $\mathbf{h}+k\mathbf{s}\in S$ for all $\mathbf{s}\in S$ and for all $k>1$. Hence, since $\{\mathbf{a}_1,\ldots,\mathbf{a}_d\}\subseteq S$, we obtain $\mathbf{h}\in \mathcal{D}(S)\setminus \operatorname{PF}(S)$, a contradiction. Therefore, $\mathcal{H}(S)=\operatorname{PF}(S)$. Now, suppose $\mathcal{H}(S)=\operatorname{PF}(S)$. Observe that $\operatorname{PF}(S)\subseteq \mathcal{D}(S)$. Let us prove that $\mathcal{D}(S)\subseteq \mathcal{H}(S)$. Let $\mathbf{a}\in \mathcal{D}(S)$. Since $\mathbf{a}+2\mathbf{a}_i$, $\mathbf{a}+2\mathbf{a}_j\in S\subseteq \mathcal{C}$ for some $i\neq j$, by Proposition~\ref{prop:outside-cone} we obtain $\mathbf{a}\in \mathcal{C}$. Moreover, by the definition of $\mathcal{D}(S)$, we have $\mathbf{a}\notin S$. So, $\mathbf{a}\in \mathcal{C}\setminus S=\mathcal{H}(S)$. Hence, $\mathcal{D}(S)\subseteq \mathcal{H}(S)$ and we have $\mathcal{H}(S)=\operatorname{PF}(S)\subseteq \mathcal{D}(S)\subseteq \mathcal{H}(S)$, which implies $\operatorname{PF}(S)=\mathcal{D}(S)$. So, $S$ is Buchsbaum.
\end{proof}

\begin{proposition}
   Let $S\subseteq \mathbb{N}^d$ be a Buchsbaum simplicial $\mathcal{C}$-semigroup, and $\mathbf{x} \in \mathcal{M}(S)$. Then $S \setminus \{\mathbf{x}\}$ is also Buchsbaum simplicial $\mathcal{C}$-semigroup. 
\end{proposition}
\begin{proof}
   Observe that when $\mathbf{x} \in \mathcal{M}(S)$, then $S \setminus \{\mathbf{x}\}$ is also a submonoid with finite complement in its rational cone, and thus a $\mathcal{C}$-semigroup. Since $S$ is Buchsbaum, we have $\mathrm{PF}(S) = \mathcal{H}(S)$. Thus, $\mathcal{H}(S \setminus \{\mathbf{x}\}) = \mathcal{H}(S) \cup \{\mathbf{x}\} = \mathrm{PF}(S) \cup \{\mathbf{x}\}.$ Observe that $\mathrm{PF}(S) \subseteq \mathrm{PF}(S\setminus  \{\mathbf{x}\})$, by the minimality of $\mathbf{x}$. Moreover, if $\mathbf{h}\in \operatorname{PF}(S\setminus \{\mathbf{x}\})$ and $\mathbf{h}\neq \mathbf{x}$, then $\mathbf{h}\in \mathcal{H}(S)=\operatorname{PF}(S)$. Thus, to prove the result, it is enough to prove that $\mathbf{x} \in \mathrm{PF}(S \setminus \{\mathbf{x}\})$. Now, since $\mathbf{x} \in S$, we have $\mathbf{x} + \mathbf{s} \in S \setminus \{\mathbf{x}\}$ for all $\mathbf{x} \in S \setminus \{\mathbf{0}\}$. Therefore, we get $\mathbf{x} \in \mathrm{PF}(S \setminus \{\mathbf{x}\})$. 
\end{proof}

\begin{corollary}
   Let $d>1$, and $S$ be a generalized numerical semigroup in $\mathbb{N}^d$. Then, $S$ is Buchsbaum if and only if $\mathcal{H}(S) = \mathrm{PF}(S).$
\end{corollary}
\begin{proof}
    Observing that every generalized numerical semigroup is a simplicial affine $\mathcal{C}$-semigroup with $\mathcal{C}=\mathbb{N}^d$, the claim is a straightforward consequence of Theorem~\ref{thm:H=PF}.
\end{proof}

Following the notation in \cite{garcia-someProperties}, if $\mathcal{C}$ is an affine integer cone in $\mathbb{N}^d$ and $A\subseteq \mathcal{C}$ is an antichain with respect to $\leq_\mathcal{C}$, we define 
$$\operatorname{I}_\mathcal{C}(A)=\{\mathbf{x}\in \mathcal{C}\mid \mathbf{x}\leq_\mathcal{C} \mathbf{a}\text{ for some }\mathbf{a}\in A\}.$$ 
Moreover a set $I\subseteq \mathcal{C}$ is an \emph{ideal} of $\mathcal{C}$ if $I+\mathcal{C}\subseteq I$. Using these definitions, we can prove the following.

\begin{theorem}\label{thm:PF-ideal}
Let $S\subseteq \mathcal{C}$, where $\mathcal{C}$ is an affine integer cone in $\mathbb{N}^d$. Then the following are equivalent:
\begin{enumerate}
\item $S$ is a $\mathcal{C}$-semigroup with $\mathcal{H}(S)=\operatorname{PF}(S)$.
\item $S=(\mathcal{C}\setminus \operatorname{I}_\mathcal{C}(A))\cup \{\mathbf{0}\}$ for some finite set $A\subseteq \mathcal{C}$, antichain with respect to $\leq_\mathcal{C}$.
\item $S=(\mathcal{M}(S)+\mathcal{C})\cup \{\mathbf{0}\}$ and $\mathcal{M}(S)$ is a multset of  $\mathcal{C}$.
\end{enumerate}
\end{theorem}
\begin{proof}
$(1) \Rightarrow (2)$ Suppose $S$ is a $\mathcal{C}$-semigroup and $\mathcal{H}(S)=\operatorname{PF}(S)$. Let $A=\mathrm{Maximals}_{\leq_\mathcal{C}}(\mathcal{H}(S))$ and we prove that $\mathcal{H}(S)=\operatorname{I}_\mathcal{C}(A)\setminus \{\mathbf{0}\}$. Observe that, by the definition of $A$, we have  $\mathcal{H}(S)\subseteq  \operatorname{I}_\mathcal{C}(A)\setminus \{\mathbf{0}\}$. Let $\mathbf{x}\in \operatorname{I}_\mathcal{C}(A)\setminus \{\mathbf{0}\}$. So, there exists $\mathbf{a}\in A$ such that $\mathbf{x}\leq_\mathcal{C} \mathbf{a}$. In particular, $\mathbf{x}+\mathbf{y}=\mathbf{a}$ for some $\mathbf{y}\in \mathcal{C}$. Assume $\mathbf{x}\in S\setminus \{\mathbf{0}\}$. In this case, since $\mathbf{a}\in \mathcal{H}(S)$ we have $\mathbf{y}\in \mathcal{H}(S)$. But since $\mathcal{H}(S)=\operatorname{PF}(S)$, we have $\mathbf{y}\in \operatorname{PF}(S)$. In particular, $\mathbf{x}+\mathbf{y}=\mathbf{a}\in S$, a contradiction. Therefore, $\mathbf{x}\in \mathcal{H}(S)$, which led to $\mathcal{H}(S)=  \operatorname{I}_\mathcal{C}(A)\setminus \{\mathbf{0}\}$.

\noindent $(2)\Rightarrow (3)$ Suppose $S=(\mathcal{C}\setminus \operatorname{I}_\mathcal{C}(A))\cup \{\mathbf{0}\}$ for some finite set $A\subseteq \mathcal{C}$ antichain with respect to $\leq_\mathcal{C}$. Since for every $\mathbf{s}\in S$ there exists $\mathbf{m}\in \mathcal{M}(S)$ such that $\mathbf{m} \leq_{\mathcal{C}} \mathbf{s}$, it follows $S\subseteq (\mathcal{M}(S)+\mathcal{C})\cup \{\mathbf{0}\}$. Moreover, if $\mathbf{m}\in \mathcal{M}(S)$ and $\mathbf{c}\in \mathcal{C}$, then $\mathbf{m}+\mathbf{c}\in S$. Otherwise $\mathbf{m}\leq_\mathcal{C}\mathbf{m}+\mathbf{c}\leq_{\mathcal{C}}\mathbf{a}$ for some $\mathbf{a}\in A$, which means $\mathbf{m}\in \operatorname{I}_\mathcal{C}(A)\setminus \{\mathbf{0}\}$, a contradiction. So, $S=(\mathcal{M}(S)+\mathcal{C})\cup \{\mathbf{0}\}$. Furthermore, $\mathcal{M}(S)$ is a multset of $\mathcal{C}$. In fact, suppose $\mathbf{r}$ is a minimal extremal ray of $\mathcal{C}$ and $\{\lambda\mathbf{r}\mid \lambda \in \mathbb{N}\}\cap S=\emptyset$. So, for all $\lambda\in \mathbb{N}$ there exists $\mathbf{a}_\lambda\in A$ such that $\lambda \mathbf{r}\leq_\mathcal{C} \mathbf{a}_\lambda$. Observing that for all $\mathbf{a}_\lambda\in \mathcal{C}$ the set $\operatorname{I}_\mathcal{C}(\{\mathbf{a}_\lambda\})$ is finite, we obtain $A$ is an infinite set, a contradiction. Therefore, there exists $m=\min \{\lambda\in \mathbb{N}\mid \lambda\mathbf{r}\in S\}$, we have $m\mathbf{r}\in \mathcal{M}(S)$ and it is an extremal ray of $\mathcal{C}$.

\noindent $(3)\Rightarrow (1)$ By hypothesis $S=(\mathcal{M}(S)+\mathcal{C})\cup \{\mathbf{0}\}$ and we can assume there exists $\{\mathbf{a}_1,\ldots,\mathbf{a}_d\}\subseteq \mathcal{M}(S)$ rays of $\mathcal{C}$ inside $S$. Observe that in this  case $S\setminus \{\mathbf{0}\}+\mathcal{C}\subseteq S\setminus \{\mathbf{0}\}$. In fact, if $\mathbf{s}\in S\setminus \{\mathbf{0}\}$ and $\mathbf{c}\in \mathcal{C}$, then there exists $\mathbf{m}\in \mathcal{M}(S)$ such that $\mathbf{m}+\mathbf{c}'=\mathbf{s}$ for some $\mathbf{c}'\in \mathcal{C}$. In particular, $\mathbf{s}+\mathbf{c}=\mathbf{m}+\mathbf{c}'+\mathbf{c}\in (\mathcal{M}(S)+\mathcal{C})=S\setminus \{\mathbf{0}\}$. This means $S\setminus \{\mathbf{0}\}$ is an ideal of $\mathcal{C}$ and, in particular, $S$ is a submonoid of the affine monoid $\mathcal{C}$. Moreover, by \cite[Proposition 6]{garcia-MED} we obtain $S$ is a $\mathcal{C}$-semigroup. Finally, let $\mathbf{h}\in \mathcal{H}(S)$ and $\mathbf{s}\in S\setminus \{\mathbf{0}\}$.  We can get $\mathbf{m}\in \mathcal{M}(S)$ such that $\mathbf{m}+\mathbf{c}=\mathbf{s}$ for some $\mathbf{c}\in \mathcal{C}$. Hence $\mathbf{h}+\mathbf{s}=\mathbf{m}+\mathbf{h}+\mathbf{c}\in (\mathcal{M}(S)+\mathcal{C}) \cup \{\mathbf{0}\}=S$. This implies that $\mathcal{H}(S)\subseteq \operatorname{PF}(S)$ and, since the reverse inclusion always holds, we have that the two sets are equal. This concludes the proof.
\end{proof}

\begin{example} \rm
Consider the affine integer cone $\mathcal{C}=\mathrm{Span}_{\mathbb{Q}_{\leq 0}}((1,1),(3,1))\cap \mathbb{N}^2$. Given the set $A=\{(5,4),(5,5),(7,4),(8,4),(10,4),(11,4)\}\subseteq \mathcal{C}$, it is not difficult to see that $A$ is an antichain with respect to $\leq_\mathcal{C}$ and it is possible to compute 
\begin{align*}   
\operatorname{I}_\mathcal{C}(A)= \left\lbrace
\begin{array}{c}
    (0,0), (1,1),(2,1),(3,1),(2,2),(3,2),(4,2),(5,2),(6,2),(3,3),(4,3), (5,3), \\ [1mm]
   (6,3),(7,3),(8,3),(9,3), (4,4),(5,4),(5,5),(7,4),(8,4),(10,4),(11,4)
\end{array}
\right\rbrace .
\end{align*} 
The set $S=(\mathcal{C}\setminus \operatorname{I}_\mathcal{C}(A))\cup \{\mathbf{0}\}$ is a $\mathcal{C}$-semigroup. Moreover, $S=(\mathcal{M}(S)+\mathcal{C})\cup \{\mathbf{0}\}$ where $\mathcal{M}(S)=\{(6,4),(6,5),(6,6),(9,4),(12,3)\}$, see also Figure~\ref{fig:ideal}.
\begin{figure}[h!]
\includegraphics[scale=0.7]{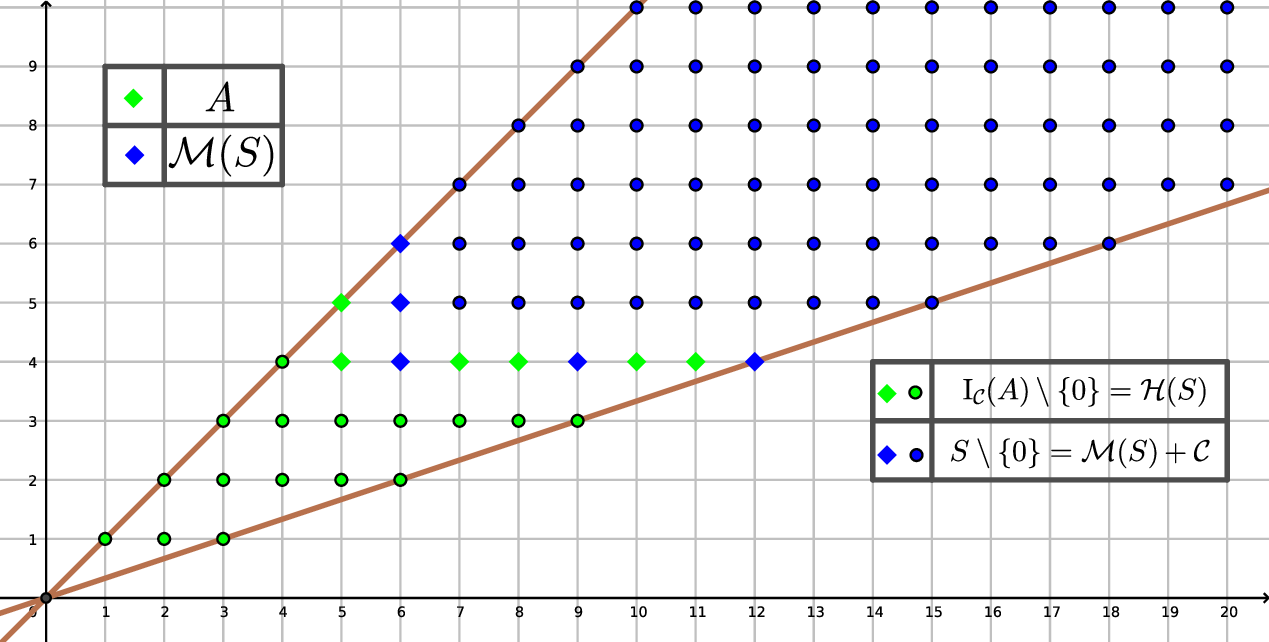}
\caption{A $\mathcal{C}$-semigroup as in Theorem~\ref{thm:PF-ideal}}
\label{fig:ideal}
\end{figure}
\end{example}

Considering the previous results, the structure of a simplicial $\mathcal{C}$-semigroup such that the related semigroup ring is Buchsbaum can be characterized. This fact is stated in the following.

\begin{corollary}\label{cor:buchs-ideal}
    Let $S\subseteq \mathbb{N}^d$ be a simplicial $\mathcal{C}$-semigrouop with $d>1$. Then $S$ is Buchsbaum if and only if $S\setminus \{\mathbf{0}\}$ is an ideal of $\mathcal{C}$. 
\end{corollary}
\begin{proof}
    In the proof of Theorem~\ref{thm:PF-ideal}, we showed that if $S=(\mathcal{M}(S)+\mathcal{C})\cup \{\mathbf{0}\}$, then $S\setminus \{\mathbf{0}\}$ is an ideal of $\mathcal{C}$. Moreover, if $S\setminus \{\mathbf{0}\}+\mathcal{C}\subseteq S\setminus \{\mathbf{0}\}$ it is trivial that $S=(\mathcal{M}(S)+\mathcal{C})\cup \{\mathbf{0}\}$. Therefore, $S\setminus \{\mathbf{0}\}+\mathcal{C}\subseteq S\setminus \{\mathbf{0}\}$ if and only if $S=(\mathcal{M}(S)+\mathcal{C})\cup \{\mathbf{0}\}$ and we can conclude by Theorem~\ref{thm:H=PF} and Theorem~\ref{thm:PF-ideal}.
\end{proof}

\section{Buchsbaum $\mathcal{C}$-semigroups having maximal embedding dimension}

Let $S\subseteq \mathbb{N}^d$ be an affine semigroup and $X$ be a non-empty subset of $S$, define
$$\mathrm{Ap}(S,X)=\{\mathbf{s}\in S\mid \mathbf{s}-\mathbf{x}\notin S\ \text{ for all } \mathbf{x}\in X\}.$$
The set $\mathrm{Ap}(S,X)$ is called the \textit{Ap\'ery set} of $S$ with respect to $X$. The set $\mathrm{Ap}(S,X)$ can be infinite, but it is known that if $E$ is a set of extremal rays of $S$ then $\mathrm{Ap}(S, E)$ is a finite set (see, for instance, \cite{jafari2024depth}). Observe that we can always assume $\mathrm{msg}(S)=E\sqcup N$ where $E$ is the set of minimal extremal rays of $S$. In this case, it is not difficult to see that $ \mathrm{msg}(S)\setminus E \subseteq \mathrm{Ap}(S, E)\setminus \{\mathbf{0}\}$. Accordingly to \cite{buchsbaum-simplicial} and \cite{garcia-MED}, we say that $S$ has \emph{minimal Ap\`ery set} or has \emph{maximal embedding dimension} if $ \mathrm{msg}(S)\setminus E = \mathrm{Ap}(S, E)\setminus \{\mathbf{0}\}$. The following result can be considered as a specialization of \cite[Proposition~20]{garcia-MED} in the case $S$ is a semigroup as in Theorem~\ref{thm:PF-ideal}. In particular, it allows us to characterize Buchsbaum simplicial $\mathcal{C}$-semigroups having maximal embedding dimension.

\begin{theorem}\label{thm:buchs-MED}
Let $\mathcal{M}$ be a multset of an affine integer cone $\mathcal{C}\subseteq \mathbb{N}^d$. Assume $\mathcal{M}=\{\mathbf{n}_1, \ldots, \mathbf{n}_d\}\sqcup \{\mathbf{m}_1,\ldots,\mathbf{m}_r\}$ such that $\{\mathbf{n}_1, \ldots, \mathbf{n}_d\}$ is a set of extremal rays of $\mathcal{C}$ and define $S=(\mathcal{M}+\mathcal{C})\cup \{\mathbf{0}\}$. Then $S$ is maximal embedding dimension if and only if for all $i,j\in \{1,\ldots,r\}$ there exists $k\in \{1,\ldots,d\}$ and $\mathbf{m}\in \mathcal{M}$ such that $\mathbf{m}\leq_\mathcal{C }\mathbf{m}_i+\mathbf{m}_j-\mathbf{n}_k$.
\end{theorem}
\begin{proof}
Let us denote $E=\{\mathbf{n}_1,\ldots, \mathbf{n}_d\}$. Observe that if $S=(\mathcal{M}+\mathcal{C})\cup \{\mathbf{0}\}$, trivially we have $\mathcal{M}
=\mathrm{Minimals}_{\leq_\mathcal{C}}(S\setminus \{\mathbf{0}\})$. Moreover, under this condition, $S$ is a $\mathcal{C}$-semigroup by \cite[Proposition 6]{garcia-MED}.  Now observe that, since $S=(\mathcal{M}+\mathcal{C})\cup \{\mathbf{0}\}$, the condition with $\mathbf{m}\in \mathcal{M}$ such that $\mathbf{m}\leq_{\mathcal{C}}\mathbf{m}_i+\mathbf{m}_j-\mathbf{n}_k$ is equivalent to $\mathbf{m}_i+\mathbf{m}_j-\mathbf{n}_k\in S$. Suppose $S$ has maximal embedding dimension. Then for all $i,j\in \{1,\ldots,r\}$ we have $\mathbf{m}_i+\mathbf{m}_j\notin \mathrm{msg}(S)$ and, as a consequence, $\mathbf{m}_i+\mathbf{m}_j\notin \mathrm{Ap}(S, E)\setminus \{\mathbf{0}\}$. So, by definition, there exists $k\in \{1,\ldots,d\}$ such that $\mathbf{m}_i+\mathbf{m}_j-\mathbf{n}_k\in S$. Now, suppose for all $i,j\in \{1,\ldots,r\}$ there exists $k\in \{1,\ldots,d\}$ such that $\mathbf{m}_i+\mathbf{m}_j-\mathbf{n}_k\in S$. In order to prove $S$ has maximal embedding dimension, it is sufficient to prove that if $\mathbf{x}\in S$ is nonzero and is not a minimal generator, then $\mathbf{x} \notin \mathrm{Ap}(S, E)$. So, assume $\mathbf{x}=\mathbf{x}_1+\mathbf{x}_2$ with $\mathbf{x}_1,\mathbf{x}_2\in S\setminus \{\mathbf{0}\}$. In particular, there exists $\mathbf{a}_1,\mathbf{a}_2\in \mathcal{M}(S)$ such that $\mathbf{a}_1\leq_{\mathcal{C}}\mathbf{x}_1$ and $\mathbf{a}_2\leq_{\mathcal{C}}\mathbf{x}_2$. Assume $\mathbf{a}_1=\mathbf{n}_i$ for some $i\in \{1,\ldots,d\}$. Then $\mathbf{x}-\mathbf{n}_i=(\mathbf{x}_1-\mathbf{a}_1)+\mathbf{x}_2\in S$, since $\mathbf{x}_2\in S$ and $\mathbf{x}_1-\mathbf{a}_1 \in \mathcal{C}$. The same argument holds in the case $\mathbf{a}_2=\mathbf{n}_i$ for some $i\in \{1,\ldots,d\}$. Hence, suppose $\mathbf{m}_i\leq_{\mathcal{C}}\mathbf{x}_1$ and $\mathbf{m}_j\leq_{\mathcal{C}}\mathbf{x}_2$ for some $i,j\in\{1,\ldots,r\}$. Let $k\in \{1,\ldots,d\}$ such that $\mathbf{m}_i+\mathbf{m}_j-\mathbf{n}_k\in S$. Therefore, for some $\mathbf{c}_1,\mathbf{c}_2\in \mathcal{C}$, we have $\mathbf{x}-\mathbf{n}_k=\mathbf{c}_1+\mathbf{c}_2+\mathbf{m}_i+\mathbf{m}_j-\mathbf{n}_k\in S$. This concludes the proof.
\end{proof}

\begin{example}\rm
Let $\mathcal{C}=\mathbb{N}^2$ and $\mathcal{M}=\{(0,4),(4,0),(2,2)\}$. With reference to Theorem~\ref{thm:buchs-MED} we can denote $\mathcal{M}=\{\mathbf{n}_1,\mathbf{n}_2\}\sqcup \{\mathbf{m}_1\}$, where $\mathbf{n}_1=(0,4)$, $\mathbf{n}_2=(4,0)$ and $\mathbf{m}_1=(2,2)$. Then $S=(\mathcal{M}+\mathbb{N}^2)\cup \{(0,0)\}$ is Buchsbaum and has maximal embedding dimension since $\mathbf{m}_1+\mathbf{m}_1-\mathbf{n}_1=(4,0) \geq_{\mathbb{N}^2} \mathbf{n}_2$.
\end{example}

Suppose $\mathcal{C}\subseteq \mathbb{N}^d$ is an affine simplicial integer cone and $E=\{\mathbf{a}_1, \ldots, \mathbf{a}_d\}$ is a set of extremal rays of $\mathcal{C}$ (also non minimal). Let $\sim$ be the relation in $\mathbb{N}^d$ defined by $\mathbf{x}\sim \mathbf{y}$ if and only if $\mathbf{x}-\mathbf{y}\in \operatorname{G}(\{\mathbf{a}_1, \ldots, \mathbf{a}_d\})=\{\sum_{i=1}^d \lambda_i\mathbf{a}_i\mid \lambda_i\in \mathbb{Z}\text{ for every}i\in \{1,\ldots,d\}\}$. It is not difficult to see that $\sim$ is an equivalence relation and if $\mathbf{x}\in \mathbb{N}^d$ we denote $[\mathbf{x}]=\{\mathbf{y}\in \mathbb{N}^d\mid \mathbf{x}\sim \mathbf{y}\}$, that is, the equivalence class of $\mathbf{x}$ with respect to $\sim$.

\begin{lemma}\label{lem:r-unique}
Let $S\subseteq \mathbb{N}^d$ be a simplicial affine semigroup and $E=\{\mathbf{a}_1, \ldots, \mathbf{a}_d\}\subseteq S$ be a set of extremal rays of $S$. Denote by $\mathcal{C}=\mathrm{cone}(E)\cap \mathbb{N}^d$. Then the following holds:
\begin{enumerate}

\item For all $\mathbf{x}\in \mathcal{C}$ there exists a unique $\mathbf{r}\in \operatorname{Ap}(\mathcal{C},E)$  and a unique vector $(\lambda_1,\ldots,\lambda_d)\in \mathbb{N}^d$ such that $\mathbf{x}\sim \mathbf{r}$ and $\mathbf{x}=\mathbf{r}+\sum_{i=1}^d \lambda_i \mathbf{a}_i$.

\item $\mathbf{x}\in \operatorname{Ap}(S,E)$ if and only if $\mathbf{x}\in \mathrm{Minimals}_{\leq_\mathcal{C}}([\mathbf{r}]\cap S)$ for some $\mathbf{r}\in \operatorname{Ap}(\mathcal{C},E)$.
\end{enumerate}
\end{lemma}
\begin{proof}
1) Let $\mathbf{x}\in \mathcal{C}$. Since $\mathcal{C}$ is a simplicial affine semigroup, by \cite[Lemma 1.4]{rosalesCohen} there exists $\mathbf{r}\in \operatorname{Ap}(\mathcal{C},E)$ such that $\mathbf{x}\sim \mathbf{r}$ and $\mathbf{x}=\mathbf{r}+\sum_{i=1}^d \lambda_i \mathbf{a}_i$, with $\lambda_1,\ldots,\lambda_d\in \mathbb{N}$. Moreover, since $\mathcal{C}$ is an integer cone, by Gordan's Lemma and a result of M. Hochster it is known that the semigroup ring $\mathbb{K}[\mathcal{C}]$ is a normal Cohen-Macaulay domain (see for instance \cite[Proposition 6.1.2]{bruns-herzog} and \cite[Theorem 6.3.5]{bruns-herzog}). Therefore, the uniqueness of the expression $\mathbf{x}=\mathbf{r}+\sum_{i=1}^d \lambda_i \mathbf{a}_i$ follows by \cite[Theorem 1.5]{rosalesCohen}.

\noindent 2) Suppose $\mathbf{x}\in \operatorname{Ap}(S,E)$ and, by the first part of this result, let $\mathbf{r}\in \operatorname{Ap}(\mathcal{C},E)$ such that $\mathbf{x}\sim \mathbf{r}$. Assume there exists $\mathbf{y}\sim \mathbf{r}$ such that $\mathbf{y}\in S$ and $\mathbf{y}\leq_{\mathcal{C}} \mathbf{x}$. Then $\mathbf{x}=\mathbf{y}+\mathbf{c}$ with $\mathbf{c}\in \mathcal{C}$ and, by transitivity we have $\mathbf{x}\sim \mathbf{y}$, that is, $\mathbf{x}=\mathbf{y}+\sum_{i=1}^d \lambda_i \mathbf{a}_i$ with $\lambda_1,\ldots,\lambda_d\in \mathbb{Z}$. Therefore, we obtain $\mathbf{c}=\sum_{i=1}^d \lambda_i \mathbf{a}_i$. Since $S$ is simplicial, the set $E$ consists of linear independent vectors and since $\mathbf{c}\in \mathcal{C}$ we obtain $\lambda_i\geq 0$ for all $i\in \{1,\ldots,d\}$. So, if there exists $\lambda_j>0$, then $\mathbf{x}-\mathbf{a}_j=\mathbf{y}+\sum_{i=1,i\neq j}^t\lambda_i\mathbf{a}_i+ (\lambda_j-1)\mathbf{a}_j\in S$, a contradiction to $\mathbf{x}\in \operatorname{Ap}(S,E)$. Hence, $\lambda_i=0$ for all $i\in \{1,\ldots,d\}$, which means $\mathbf{x}=\mathbf{y}$. Conversely, suppose $\mathbf{x}\in \mathrm{Minimals}_{\leq_\mathcal{C}}\{\mathbf{y}\in S\mid \mathbf{y}\sim \mathbf{r}\}$ with $\mathbf{r}\in \operatorname{Ap}(\mathcal{C},E)$. Assuming $\mathbf{x}\notin \operatorname{Ap}(S,E)$, we obtain $\mathbf{x}-\mathbf{a}_i\in S$ for some $i\in \{1,\ldots,d\}$. Moreover, $\mathbf{x}-\mathbf{a}_i \sim \mathbf{x}$ and $\mathbf{x}-\mathbf{a}_i\leq_\mathcal{C}\mathbf{x}$, which contradicts the minimality of $\mathbf{x}$. So, $\mathbf{x}\in \operatorname{Ap}(S,E)$.
\end{proof}

\begin{remark}\rm
Assume $S$ is a generalized numerical semigroup, that is, $\mathcal{C}=\mathbb{N}^d$. Denoting by $\mathbf{e}_1,\ldots,\mathbf{e}_d$ the standard basis vectors in $\mathbb{N}^d$, for each $i\in \{1,\ldots,d\}$ we can express the extremal ray as $\mathbf{a}_i=a_i\mathbf{e}_i$ for some $a_i\in \mathbb{N}$. Therefore, in this case $\operatorname{Ap}(\mathcal{C},E)=\{(n_1,\ldots,n_d)\in \mathbb{N}^d\mid n_i<a_i\text{ for all }i\in \{1,\ldots,d\}\}$. Moreover, if $\mathbf{x}=(x_1,\ldots,x_d)\in \mathbb{N}^d$ and $\mathbf{x}\sim \mathbf{r}$ with $\mathbf{r}\in \operatorname{Ap}(\mathcal{C},E)$, then $\mathbf{r}=(r_1,\ldots,r_d)$ with $0\leq r_i<a_i$ and $r_i \equiv x_i \mod a_i$ for all $i\in \{1,\ldots,t\}$. In the very special case that $S$ is a numerical semigroup, we have $\mathcal{C}=\mathbb{N}$, $E=\{m\}$ and $\operatorname{Ap}(\mathcal{C},E)=\{0,1,\ldots,m-1\}$. By these observations, the set $\operatorname{Ap}(\mathcal{C},E)$ can be seen as the set of \emph{representatives} of the element in $\mathcal{C}$ with respect to the relation $\sim$, and Lemma~\ref{lem:r-unique} is a generalization of \cite[Lemma 2.4]{ns-book}, a well-known property in the context of numerical semigroups.
\end{remark}

\begin{lemma}\label{lem:ap-Csem}
Let $S\subseteq \mathbb{N}^d$ be a simplicial affine semigroup and $E=\{\mathbf{a}_1, \ldots, \mathbf{a}_d\}\subseteq S$ be a set of extremal rays of $S$. Denote by $\mathcal{C}=\mathrm{cone}(E)\cap \mathbb{N}^d$. For every $\mathbf{r}\in \operatorname{Ap}(\mathcal{C},E)$ the following holds:
\begin{enumerate}
\item If $\mathbf{r}\in S$, then $\operatorname{Ap}(S,E)\cap [\mathbf{r}]=\{\mathbf{r}\}$.
\item If $S$ is a $\mathcal{C}$-semigroup and $\mathbf{r}\notin S$, then $\left \lvert\operatorname{Ap}(S,E)\cap [\mathbf{r}]\right \rvert \geq d$.
\end{enumerate}
\end{lemma}
\begin{proof}
Let $\mathbf{r}\in \operatorname{Ap}(\mathcal{C},E)$. Suppose first that $\mathbf{r}\in S$. Then it is trivial that $\mathbf{r}\in \operatorname{Ap}(S,E)$, since $S\subseteq \mathcal{C}$. Moreover, assume there exists $\mathbf{x}\in \operatorname{Ap}(S,E)$ such that $\mathbf{x}\sim \mathbf{r}$. Since $\mathbf{x}\in \mathcal{C}$, by Lemma~\ref{lem:r-unique} we have $\mathbf{x}=\mathbf{r}+\sum_{i=1}^d \lambda_i \mathbf{a}_i$ with $\lambda_1,\ldots,\lambda_d\in \mathbb{N}$. If there exists $j\in \{1,\ldots,d\}$ such that $\lambda_j>0$, then $\mathbf{x}-\mathbf{n}_j=\mathbf{r}+\sum_{i=1,i\neq j}\lambda_i\mathbf{a}_i+(\lambda_j-1)\mathbf{a}_j\in S$, contradicting $\mathbf{x}\in \operatorname{Ap}(S,E)$. Therefore, we obtain $\mathbf{x}=\mathbf{r}$ and, as a consequence, if $\mathbf{r}\in S$, then $\operatorname{Ap}(S,E)\cap [\mathbf{r}]=\{\mathbf{r}\}$. Now suppose $S$ is a $\mathcal{C}$-semigroup and $\mathbf{r}\notin S$. Then, for all $i\in \{1,\ldots,d\}$ we can define $\alpha_i=\min\{\lambda\in \mathbb{N}\mid \mathbf{r}+\lambda \mathbf{a}_i\in S\}$. So, it is not difficult to see that, for every $i\in \{1,\ldots,d\}$, we have $\mathbf{r}+\alpha_i\mathbf{a}_i \in \operatorname{Ap}(S,E)$ and trivially $(\mathbf{r}+\alpha_i\mathbf{a}_i)\sim \mathbf{r}$. Therefore, we obtain $\left \lvert\operatorname{Ap}(S,E)\cap [\mathbf{r}]\right \rvert \geq d$. 
\end{proof}

\begin{example}\rm
It can occur that the inequality in (2) of Lemma~\ref{lem:ap-Csem} is not sharp. For instance,  
consider the affine semigroup $S=\langle (0,3),(0,4),(0,5),(3,0),(4,0),(5,0),(1,2),(3,1)\rangle $. It is possible to verify that $S$ is a generalized numerical semigroup having the following set of gaps:
\begin{align*}   
 \mathcal{H}(S)=\left\lbrace
\begin{array}{c}
   (0, 1), (0, 2), (1, 0), (1, 1), (1, 3), (1, 4), (2, 0), (2, 1),\\ (2, 2), (2, 3), (2, 5), (2, 6), (3, 2), (4, 1), 
  (5, 1)
\end{array}
\right\rbrace .
\end{align*} 
In this case, $E=\{(0,3),(3,0)\}$ is a set of extremal rays of $S$ and $(7,1),(1,7),(4,4) \in \operatorname{Ap}(S,E)$, where
each one belongs to the equivalence class of $(1,1)\in \operatorname{Ap}(\mathbb{N}^2,E)$.
\end{example}

The relation between ideals of an integer affine cone and Buchsbaum simplicial $\mathcal{C}$-semigroup, expressed in Corollary~\ref{cor:buchs-ideal}, allows to provide the following result about minimal presentations, which gives an explicit expression of the formula in \cite[Theorem 13]{buchsbaum-simplicial} for $\mathcal{C}$-semigroups.

\begin{theorem}\label{thm:mu-MED}
Let $S\subseteq \mathbb{N}^d$ be a Buchsbaum simplicial affine semigroup having maximal embedding dimension. Assume $\mathrm{msg}(S)=\{\mathbf{a}_1, \ldots, \mathbf{a}_d\}\sqcup \{\mathbf{a}_{d+1},\ldots,\mathbf{a}_{d+m}\}$, where $E=\{\mathbf{a}_1, \ldots, \mathbf{a}_d\}$ is a set of extremal rays of $S$. If $S$ is a $\mathcal{C}$-semigroup, then
$$\mu(I_S) = \frac{1}{2}\left(m(m+1)+d(d-1)\operatorname{g}(S)\right).$$
\end{theorem}
\begin{proof} If $d=1$, then $S$ is a maximal embedding dimension numerical semigroup, up to isomorphism. In this case, the result follows from \cite[Theorem 8.30]{ns-book}. So, assume $d\geq 2$. Since $S$ is Buchsbaum and $S$ has maximal emebdding dimension (or minimal Ap\'ery set), by \cite[Theorem 13]{buchsbaum-simplicial} and with reference to the equivalence relation $\sim$, we have $\mu(I_S)=\frac{m(m+1)}{2}+\lambda \frac{d(d-1)}{2}$, where 
$$\lambda=|\left \lbrace[\mathbf{x}]\mid \mathbf{x}
\in \mathcal{C}, \left \lvert\operatorname{Ap}(S,E)\cap [\mathbf{x}]\right \rvert = d\right \rbrace |.$$
So, it is sufficient to prove that $\lambda=\operatorname{g}(S)$. If $\mathbf{x}\in \mathcal{C}$, by Lemma~\ref{lem:r-unique} $[\mathbf{\mathbf{x}}]=[\mathbf{r}]$ for a unique $\mathbf{r}\in \operatorname{Ap}(\mathcal{C},E)$. Moreover, since $S$ is Buchsbaum, if $\mathbf{r}\notin S$, then by Lemma~\ref{lem:ap-Csem} and \cite[Theorem 9]{buchsbaum-simplicial} we obtain $|\operatorname{Ap}(S,E)\cap [\mathbf{r}]|=d$. While, again by Lemma~\ref{lem:ap-Csem}, if $\mathbf{r}\in S$, then the class $[\mathbf{r}]$ does not contribute to computation of $\lambda$. Therefore, it follows $\lambda=|\operatorname{Ap}(\mathcal{C},E)\setminus S|$. In order to conclude, we show that $\operatorname{Ap}(\mathcal{C},E)\setminus S=\mathcal{H}(S)$. The inclusion $\operatorname{Ap}(\mathcal{C},E)\setminus S \subseteq \mathcal{H}(S)$ is trivial. For the converse, let $\mathbf{h}\in \mathcal{H}(S)$ and assume $\mathbf{h}\notin \operatorname{Ap}(\mathcal{C},E)$. Hence, there exists $\mathbf{a}_i$ such that $\mathbf{h}-\mathbf{a}_i=\mathbf{c}\in \mathcal{C}$. By Corollary~\ref{cor:buchs-ideal}, $S\setminus \{\mathbf{0}\}$ is an ideal of $\mathcal{C}$. Hence, since $\mathbf{a}_i\in S$, we obtain $\mathbf{h}=\mathbf{a}_i+\mathbf{c}\in S$, a contradiction. Therefore, $\mathbf{h}\in \operatorname{Ap}(\mathcal{C},E)$ and this conclude the proof.
\end{proof}

\begin{example}\rm

Let $\mathcal{C}=\mathrm{Span}_{\mathbb{Q}_{\geq 0}}(\{(1,1),(3,1)\})\cap \mathbb{N}^2=\langle (1,1), (2,1),(3,1)\rangle$ and consider the $\mathcal{C}$-semigroup $S=\mathcal{C}\setminus \{(1,1),(2,1)\}$.  In this case 
$$\mathrm{msg}(S)= \{ (2,2),(3,1)\} \cup \{(3,2),(3,3),(4,2),(4,3),(5,2) \}$$
where $\{ (2,2),(3,1)\}$ is the set of minimal extremal rays of $S$. By Theorem~\ref{thm:buchs-MED}, it is possible to verify that $S$ has maximal embedding dimension. We can compute a minimal presentation of $S$ and its cardinality with the help of a computer algebra software. For instance, we can perform this computation by the \texttt{GAP} (\cite{gap}) package \texttt{numericalsgps} (\cite{numericalsgps}), in the following way:
\begin{verbatim}
gap> G:=[[2,2],[3,1],[3,2],[3,3],[4,2],[4,3],[5,2]];;
gap> s:=AffineSemigroup(G);
<Affine semigroup in 2 dimensional space, with 7 generators>
gap> m:=MinimalPresentationOfAffineSemigroup(s);;
gap> Length(m);
17
\end{verbatim}  
So, $\mu(I_S)=17$ and, with reference to the notation of Theorem~\ref{thm:mu-MED}, we have $m=5$, $t=2$ and $\operatorname{g}(S)=2$. In fact, $17=(5\cdot 6+2\cdot 2)/2$.
\end{example}

Interestingly, Theorem~\ref{thm:mu-MED} gives a very nice relationship between the first and last Betti number of the semigroup ring associated to a Buchsbaum maximal embedding dimension simplicial $\mathcal{C}$-semigroup. Since the projective dimension of a semigroup ring associated to $\mathcal{C}$-semigroup minimally generated by $\{\mathbf{a}_1, \ldots, \mathbf{a}_n\}$ is $n-1$, we will denote the first and last Betti number by $\beta_1$ and $\beta_{n-1}$ respectively, 

\begin{corollary}
    Let $S = \langle \mathbf{a}_1, \ldots, \mathbf{a}_n \rangle \subseteq \mathbb{N}^d$ be a Buchsbaum simplicial $\mathcal{C}$-semigroup. If $S$ has maximal embedding dimension, then
$$\beta_{n-1} = \frac{2 \beta_1- (n-d)(n-d+1)}{d(d-1)}.$$
\end{corollary}
\begin{proof}
Since $S$ is Buchsbaum, by Theorem~\ref{thm:H=PF}, we get $\operatorname{g}(S) = \vert \mathrm{PF}(S) \vert$. Also, since $S$ has maximal projective dimension, by \cite[Corollary 7]{garcia-pseudofrobenius}, we get $\vert \mathrm{PF}(S) \vert = \beta_{n-1}$. Therefore, we get $\operatorname{g}(S) = \beta_{n-1}$. Since, $\mu(I_S)$ correspond to $\beta_1$ and $S$ has maximal embedding dimension, by Theorem~\ref{thm:mu-MED}, we get
$$\beta_1 = \frac{1}{2}[(n-d)(n-d+1) + d(d-1)\beta_{n-1}].$$
Solving this, we get the desired equality.
\end{proof}

\section{Concluding remarks}

\begin{definition}[{\cite[Theorem 1.4]{rosales97}, \cite[Proposition 2.2]{hemagluingres}}] \rm
	Let $S \subseteq \mathbb{N}^r$ be an affine semigroup. Let $A$ be the minimal generating set of $S$ and $A = A_1 \amalg A_2$ be a nontrivial partition of $A$. Let $S_i$ be the submonoid of $\mathbb{N}^d$ generated by $A_i, i \in {1, 2}$. Then $S = S_1 + S_2.$ We say that $S$ is the gluing of $S_1$ and $S_2$ by $\mathbf{d}$ if the following two conditions hold:
	\begin{enumerate}
		\item[(i)] $\mathbf{d} \in S_1 \cap S_2$ and,
		\item[(ii)] $G(S_1 ) \cap G(S_2) = \mathbf{d}\mathbb{Z}$. 
	\end{enumerate}
	 Moreover, let $S$ be a gluing of $S_1 = \langle \mathbf{a}_1, \ldots, \mathbf{a}_n \rangle$ and $S_2 = \langle \mathbf{b}_1, \ldots, \mathbf{b}_{n'} \rangle$ by $\mathbf{d}$. Let $\mathbb{K}[S_1] \cong R_{S_1}/I_{S_1}$ and $\mathbb{K}[S_2] \cong R_{S_2}/I_{S_2}$, where $R_{S_1} = \mathbb{K}[x_1\ldots, x_n]$ and $R_{S_2} = \mathbb{K}[y_1, \ldots,y_{n'}]$ are polynomial rings over $\mathbb{K}$ on different set of variables. Then the ideal $I_S \subseteq R = \mathbb{K}[x_1,\ldots,x_n,y_1,\ldots,y_{n'}]$ is minimally generated by the minimal generators of $I_{S_1}$ and $I_{S_2}$ and a binomial $\rho$. The binomial $\rho$ is defined as $\rho = \mathbf{x^u} - \mathbf{y^v}$ where $\mathbf{u} = (u_1, \ldots, u_n)$ and $\mathbf{v} = (v_1,\ldots,v_{n'})$ such that $\mathbf{d} = \sum_{i=1}^n u_i\mathbf{a_i} = \sum_{i=1}^{n'}v_i\mathbf{b_i}.$
\end{definition}  

\begin{remark}\rm
Let $S$ be a gluing of two affine semigroups. Then $S$ is a complete intersection if and only if $S$ is a gluing of two complete intersection affine semigroups (\cite[cf. Theorem 3.1]{fischermorrisshapiro}]). Also, it is proved that $\mathbb{K}[S]$ is Cohen-Macaulay if and only if $\mathbb{K}[S_1]$ and $\mathbb{K}[S_2]$ are Cohen-Macaulay, $\mathbb{K}[S]$ is Gorenstein if and only if $\mathbb{K}[S_1]$ and $\mathbb{K}[S_2]$ are Gorenstein (\cite[Theorem 1.5]{hemagluing}), $\mathbb{K}[S]$ is of maximal projective dimension if and only if $\mathbb{K}[S_1]$ and $\mathbb{K}[S_2]$ are of maximal projective dimension (\cite[Theorem 2.6]{affineI}, \cite[Proposition 6.6]{affineII}). Unlike these properties, we show by an example that the Buchsbaum property is not preserved under gluing. 
\end{remark}

\begin{example}\rm \textbf{Gluing of two Buchsbaum affine semigroups may not be Buchsbaum:}
Consider the affine semigroup
\[
S = \langle (0,7), (14,0), (21,0), (7,7), (6,12), (8,16) \rangle.
\]
Decompose $S$ into two subsemigroups
\[
S_1 = \langle (0,7), (14,0), (21,0), (7,7) \rangle \quad \text{and} \quad 
S_2 = \langle (6,12), (8,16) \rangle.
\]
Now, one can observe that $S$ is a gluing of $S_1$ and $S_2$ along the element $(14,28)$. Now, define the semigroup
\[
S_3 = \langle (0,1), (2,0), (3,0), (1,1) \rangle.
\]
Note that $S_1 = 7 \cdot S_3$, and hence the corresponding semigroup rings $\mathbb{K}[S_1]$ and $ \mathbb{K}[S_3]$ are isomorphic. For $S_3$, we have $\mathcal{H}(S_3) = \mathrm{PF}(S_3) = \{(1,0)\}.$ Therefore, by Theorem~\ref{thm:H=PF}, the semigroup ring $\mathbb{K}[S_3]$ is Buchsbaum. Since $\mathbb{K}[S_1] \cong \mathbb{K}[S_3]$, we conclude that $\mathbb{K}[S_1]$ is also Buchsbaum.
On the other hand, $S_2$ is a one-dimensional semigroup, and any one-dimensional semigroup ring is Buchsbaum. Thus, $\mathbb{K}[S_2]$ is Buchsbaum as well.
Now consider $S$. Its minimal extremal rays are $\{(0,7), (14,0)\}$. Consider the element $(15,16)$. Observe that $(15,16) \in G(S)$, the group generated by $S$. Also, 
\[
(15,16) + 2(0,7) = (15,30) \in S, \quad \text{and} \quad (15,16) + 2(14,0) = (43,16) \in S.
\]
Therefore, $(15,16)$ belongs to $\mathcal{D}(S)$. However, by \cite[Theorem~2.6]{affineI}, we have $\mathrm{PF}(S) = \{(31,48)\}.$ Thus, $(15,16) \in \mathcal{D}(S)$ but $(15,16) \notin \mathrm{PF}(S)$. By Proposition~\ref{prop:buchsbaum}, this implies that $\mathbb{K}[S]$ is \textbf{not} Buchsbaum.
\end{example}

This instantly raise a natural and important question:

\begin{question}  
Let $S$ be an affine semigroup obtained by gluing of two Buchsbaum affine semigroups $S_1$ and $S_2$.  
Give a necessary and sufficient condition for $S$ to be Buchsbaum in terms of $S_1$ and $S_2$.
\end{question}

  An algebraic generalization of the Buchsbaum property for commutative rings was given by Trung in \cite{trung-gen}, where the author introduced the concept of the \emph{Generalized Cohen-Macaulay} property. To the best of our knowledge, there are no known results that relate the Generalized Cohen-Macaulay property to semigroup rings, analogous to the known characterizations of Buchsbaum, Cohen-Macaulay, and Gorenstein semigroup rings. Apart from a possible general characterization of Generalized Cohen–Macaulay affine semigroups, we ask whether there exist simplicial $\mathcal{C}$-semigroups $S$ such that $\mathbb{K}[S]$ is Generalized Cohen-Macaulay but not Buchsbaum.

\begin{question}
    Does there exist a Generalized Cohen Macaulay simplicial $\mathcal{C}$-semigroup which is not Buchsbaum?
\end{question}

{\it Acknowledgement.} The first author would like to thank Chennai Mathematical Institute and Infosys Foundation for financial support. The second author is a member of the group GNSAGA of Istituto Nazionale di Alta Matematica (INdAM), Italy.

\bibliographystyle{plain}
\bibliography{biblio-2}

	\end{document}